\documentclass[10pt]{amsart}
\usepackage{graphicx, amsmath, amssymb, amsthm, amsfonts, mathtools, tikz-cd, comment, tikz, caption, subcaption, url, tabularx, enumitem, mathrsfs, amsrefs, appendix, float}

\newtheorem{introthm}{Theorem}
\newtheorem{thm}{Theorem}[section]
\newtheorem{lem}[thm]{Lemma}

\newtheorem{cor}[thm]{Corollary}

\theoremstyle{definition}
\newtheorem{defn}[thm]{Definition}

\newcommand{\tF}{t_\mathscr{F}(x)}

\newcommand{\Proj}{\mathrm{Proj}}
\newcommand{\Fp}{\mathscr{F}_\varphi}
\newcommand{\Pp}{\mathscr{P}_\varphi}
\newcommand{\Xp}{\mathscr{X}_\varphi}
\newcommand{\U}{\mathscr{U}}

\title{An Infinite Transitivity Theorem}
\author{Miles Gould}
\date{\today}

\begin{document}

\begin{abstract}
    In this note, we promote an infinite Kadison transitivity theorem on massive $C^*$-algebras, including the Calkin algebra. This transitivity stems from the analog of countable degree-1 saturation on pure states which is inherited from these algebras via excision. We show this saturation to be equivalent to several order-theoretic properties on the quantum filter associated to the state, in particular the property of being a quantum P-point. While we show their existence is independent from ZFC, under basic set theoretic assumptions, we produce a plethora of these states. Finally, we find an irreducible representation of the Calkin algebra which fails infinite transitivity.
\end{abstract}

\maketitle

\section{Introduction}

\subsection{Motivation}

The Kadison transitivity theorem is among the most fundamental and classical theorems in $C^*$-algebra theory, and is the basis for many results in the representation theory of $C^*$-algebras.

Since the advent of continuous model theory in 2008 \cite{YA08}, its applications in operator algebras have proliferated. In 2017 \cite{FA17}, it was shown by Farah that, in the theory of irreducible representations of $C^*$-algebras, ultraproducts admit an infinite version of the Kadison transitivity theorem. Here we show this infinite transitivity also applies to certain irreducible representations on corona algebras $\mathscr{Q}(A)=\mathscr{M}(A)/A$ associated to nonunital $\sigma$-unital $C^*$-algebras $A$, in particular, the Calkin algebra $\mathscr{Q}(H)$. These algebras were first shown to have a (not so) weak saturation called countable degree-1 saturation by Farah and Hart in 2011 \cite{FA11}.

Around the same time, in 2014 \cite{FA14}, Farah and Weaver formalized the notion of maximal quantum filters, the quantum analog of ultrafilters. These filters assisted in the study of excision, which allows one to extract the value of a pure state $|\varphi(b)|$ on an operator $b$ using the quantities $\|aba\|$ for $a$ in the quantum filter associated to $\varphi$. Our version of infinite transitivity relies on a combination of a strong form of excision and countable degree-1 saturation.

\subsection{Outline}

We say a representation $\pi:C\rightarrow\mathscr{B}(H)$ of a $C^*$-algebra $C$ admits \textit{infinite transitivity} if for every separable subspace $K$ of $H$, the map $C\rightarrow \mathscr{B}(K)$ by $c\mapsto P_K\pi(c)P_K$ is onto, where $P_K$ is the projection onto $K$. This is equivalent to the statement that there is always an operator in $C$ sending a given orthonormal sequence to another given bounded sequence in $H$. Since the Calkin algebra is countably degree-1 saturated, it is reasonable to ask: Does every irreducible representation of the Calkin algebra have infinite transitivity? We produce a negative answer.

\begin{introthm}\label{counter}
    There exists an irreducible representation of $\mathscr{Q}(H)$ which does not admit infinite transitivity.
\end{introthm}

Of course, the next question is: Does any representation admit infinite transitivity? Under basic set-theoretic assumptions, in particular the continuum hypothesis, we produce a positive answer.

\begin{introthm}
    Under the continuum hypothesis, there exists an irreducible representation of $\mathscr{Q}(H)$ which admits infinite transitivity.
\end{introthm}

To accomplish this, we define the countable degree-$1$ saturation (def. \ref{normsat} and \ref{deg1}) of a $C^*$-algebra with a state, $(C,\varphi)$, and deduce various consequences of it. In particular, for such a structure with pure $\varphi$, the quantum filter (def. \ref{filt}) associated is a P-filter (def. \ref{Pstuff}). Thanks to the strong excision granted by these quantum filters, we show that many of these order-theoretic consequences of this saturation are, in fact, equivalent to it (theorem \ref{satequiv}). As such, we call these saturated states P-states. Then we show that the GNS-representation associated to a P-state, called a P-representation, has infinite Kadison transitivity.

\begin{introthm}\label{iktt}
    Let $A$ be a nonunital, $\sigma$-unital $C^*$-algebra, and $\pi:\mathscr{Q}(A)\rightarrow\mathscr{B}(H)$ a P-representation. Then $\pi$ admits infinite transitivity. In particular, for every separable subspace $K$ of $H$ and $T\in\mathscr{B}(K)$ self-adjoint (positive, unitary, resp.), there exists $c\in\mathscr{Q}(A)$ such that $\|c\|=\|T\|$, $P_K\pi(c)P_K=T$, and $c$ is self-adjoint (positive, unitary, resp.).
\end{introthm}

This follows from infinite Glimm-Kadison on direct sums of P-representations.

\begin{introthm}\label{igk}
    Let $A$ be a nonunital, $\sigma$-unital $C^*$-algebra, and $\pi:\mathscr{Q}(A)\rightarrow\mathscr{B}(H)$ a direct sum of countably many inequivalent P-representations $\pi_j:\mathscr{Q}(A)\rightarrow\mathscr{B}(H_j)$ and $K_j$ a separable subspace of $H_j$. Then the map $\mathscr{Q}(A)\rightarrow\prod_j\mathscr{B}(K_j)$ by $c\mapsto P_K\pi(c)P_K$ is onto, where $K=\bigoplus_jK_j$.
\end{introthm}

As in theorem \ref{iktt}, we can choose $c$ as we like. With the theory complete, under the existence of a P-point, we produce a P-state, on the corona algebra $\mathscr{Q}(A)$ of various nonunital, $\sigma$-unital $A$. However, as shown in \cite{BI13}, there is a forcing in which these states cannot exist.

Finally, we find an irreducible representation of the Calkin algebra which fails infinite transitivity, granting theorem \ref{counter}. Compare this situation with that of \cite{FA10}, wherein the question of whether the commutant of $\mathscr{B}(H)$ in $\mathscr{B}(H)^\U$ is trivial also depends on the ultrafilter $\U$. In our case, for a diagonal pure state on $\mathscr{Q}(H)$ (or equivalently $\mathscr{B}(H)$), whether the GNS-representation admits infinite transitivity depends on the corresponding ultrafilter $\U$. Another recent (and very successful) instance of this ultrafilter-dependence phenomenon was in \cite{ZA25}.

\subsection{Acknowledgements}

I would like to thank Tristan Bice and Beatriz Zamora-Avilés for their useful insight into various proofs modified throughout, and Alan Dow for the key to finding a counterexample to infinite transitivity. I give my gratitude to Charles Akemann and David Gao, not only for their help with this paper, but for their overall teaching and mentorship. I also thank Juris Steprāns and Michael Carey for several valuable conversations about set and model theory as I was writing this paper. Of course, I am indebted to Ilijas Farah for his consistent prowess and for teaching me the beauty of combinatorial set theory and massive $C^*$-algebras. Without his existing literature, initial prediction, and advisory, this paper would never have come to fruition.

\subsection{Preliminaries}

Throughout, $A$ denotes an arbitrary $C^*$-algebra and $\varphi$ denotes an arbitrary state on $A$. However, when dealing with large $C^*$-algebras (like $\mathscr{Q}(A)$) we will denote them by $C$ to avoid confusion with $A$. We denote finite subsets $F\subseteq X$ by $F\Subset X$.

Before understanding saturation of states, we need to understand degree-1 saturation of $C^*$-algebras, an extremely powerful property enjoyed by massive corona algebras.

\begin{defn}\label{normsat}
    A \textit{degree-1 $*$-polynomial} $P(\bar{x})$ over $C$ in variables $\bar{x}=(x_1,\dots,x_m)$ is an expression of the form
    \[P(\bar{x})=\sum_{i\leq m,j\leq n}(a_{ij1}x_ja_{ij2}+a_{ij3}x_j^*a_{ij4})+a_0,\]
    where $a_0\in C$, $a_{ijk}\in \tilde{C}$, $n\in\mathbb{N}$. A \textit{degree-$1$ condition of }$C$ is an expression of the form
    \[\|P(\bar{x})\|=r,\]
    where $P(\bar{x})$ is a degree-$1$ $*$-polynomial and $r\geq0$. A \textit{countable degree-$1$ type} $t(\bar{x})$ is collection of degree-$1$ norm conditions, i.e.
    \[t(\bar{x})=\{\|P_j(\bar{x})\|=r_j:j\in\mathbb{N}\}.\]
    We call $t(\bar{x})$ \textit{satisfiable} if for every $F\Subset \mathbb{N}$, $\epsilon>0$, there exists $\bar{b}$ in $C$ such that $\max_{j\in F}|\|P_j(\bar{b})\|-r_j|<\epsilon$. Similarly, $t(\bar{x})$ is \textit{realized} if there exists $\bar{b}$ such that $\|P_j(\bar{b})\|=r_j$ for all $j\in\mathbb{N}$. We say that $C$ is \textit{countably degree-$1$ saturated} if every satisfiable countable degree-$1$ type is realized.
\end{defn}

The most exotic examples of countably degree-1 norm saturated $C^*$-algebras are those whose saturation does not follow from the standard model-theoretic constructions, namely non-ultraproduct corona algebras. To define coronas, we follow \cite[13]{FA19}.

\begin{defn}\label{corona}
    For a nonunital $C^*$-algebra $A$, we define the \textit{strict topology on} $A$ by convergence $a_\lambda\rightarrow a$ iff $\lim_\lambda\|(a_\lambda-a)h\|=\lim_\lambda\|h(a_\lambda-a)\|=0$ for all $h\in A$. We define the \textit{multiplier algebra} $\mathscr{M}(A)$ as the completion of $A$ in the strict topology, with algebraic operations defined naturally on the Cauchy classes. Note that this is equivalent to the standard definitions of the multiplier algebra.

    The \textit{corona algebra} $\mathscr{Q}(A)$ is the quotient $C^*$-algebra $\mathscr{Q}(A)=\mathscr{M}(A)/A$ with quotient map $\pi:\mathscr{M}(A)\rightarrow\mathscr{Q}(A)$.
\end{defn}

When $A$ is nonunital yet $\sigma$-unital, remarkably $\mathscr{Q}(A)$ is countably degree-1 saturated, and a concise proof can be found in \cite[15.1.5]{FA19}.

\section{Saturated States}

Throughout this section, we define countably degree-1 saturated states and work towards understanding them.

\begin{defn}\label{deg1}
    A \textit{degree-$1$ condition of}$(C,\varphi)$ is an expression $\rho(P(\bar{x}))=s$, where $P$ is a degree-$1$ $*$-polynomial, $s\in\mathbb{C}$, and $\rho $ is either $\|\cdot\|$ or $\varphi$. A \textit{countable degree-$1$ type} is a countable collection of degree-$1$ conditions, i.e.
    \[
t(\bar{x})=\{\rho_j(P_j(\bar{x}))=s_j:j\in\mathbb{N}\}.\]
    We define the satisfiability and realizability of $t(\bar{x})$ as in \ref{normsat}. We also define the countable degree-1 saturation of $(C,\varphi)$ analogously.
\end{defn}

In the remainder of this section, we will work towards the remarkable property of pure states which imbues them with this saturation.

\subsection{Quantum Filters}

\begin{defn}\label{filt}
    Given a state $\varphi$ on a $C^*$-algebra $A$, the \textit{quantum filter} associated to $\varphi$ is the collection
    \[\Fp=\{a\in A_{+,1}:\varphi(a)=1\},\]
    where $A_{+,1}$ are the positive, norm 1 elements of $A$.
    The \textit{projection filter} associated to $\varphi$ is the collection
    \[\Pp=\{p\in \Proj(A):\varphi(p)=1\}.\]
    We equip these collections with the order $\ll$ on $A_{+,1}$ by
    \[a\ll b\text{ iff }ab=a \text{ (equiv. }ba=a).\]
    Note that, if $p\in\Proj(A)$, $a\in A_{+,1}$, then $p\ll a$ iff $p\leq a$, the usual order on $A_{+,1}$.

    There is also an abstract definition of a quantum filter, namely a \textit{quantum filter in $A$} is a subset $\mathscr{F}\subseteq A_{+,1}$ with the property that for any $a_1,\dots,a_n\in\mathscr{F}$, $\|a_1\cdots a_n\|=1$. A \textit{maximal quantum filter} is a quantum filter not properly contained in any other.

\end{defn}

Caution: Quantum filters are not always filters in the order-theoretic sense. More precisely, they are rarely downwards directed under $\ll$ or $\leq$, though interestingly they are always very close \cite[5.3.4]{FA19} under both orderings. Thankfully, the quantum filters studied in this note will be true (P-)filters under $\ll$.

The beauty of maximal quantum filters is their natural bijective correspondence with pure states, i.e. $\varphi\mapsto\Fp$ is a bijection from pure states to maximal quantum filters \cite[5.3.8]{FA19}. However, the map $\varphi\mapsto\Fp$ from states to quantum  filters is almost never injective or surjective. For example, $\{1\}$ is always a quantum filter, but is only of the form $\Fp$ when $A=\mathbb{C}$. In fact, the more natural correspondence is between faces of the state space and quantum filters, though there are still sometimes quantum filters not corresponding to a face of the state space.

\subsection{Excision, P-Excision, and Quantum P-points}

Quantum filters were first introduced in their relation to excision, a powerful property wherein the value of $|\varphi(b)|$ can be approximated by quantities $\|aba\|$ for  $a\in A_{+,1}$.

\begin{defn}
    For a state $\varphi$ on a $C^*$-algebra $A$ and subset $\mathscr{F}\subseteq A_{+,1}$, $X\subseteq A$, we say that $\mathscr{F}$ excises $\varphi$ on $X$ if for all $b\in X$,
    \[\inf_{a\in\mathscr{F}}\|aba-\varphi(b)a^2\|=0.\]
\end{defn}

Most notably, when $\varphi$ is pure, $\mathscr{F}_\varphi$ excises $\varphi$ on the entirety of $A$ \cite[5.2.1]{FA19}. However, for small subsets $X\subseteq A$, we can usually excise $\varphi$ using a small portion of $A_{+,1}$. In fact, if $X$ is separable, then there is always some countable $\mathscr{F}\subseteq\Fp$ excising $\varphi$ on $X$. This fact was cleverly leveraged in \cite[15.3.8]{FA19} to demonstrate that, when $C$ is countably degree-1 saturated, for every separable $C^*$-subalgebra $B$ of $C$, there exists a single $a\in C_{+,1}$ excising $\varphi$ on $B$. This property is quite remarkable, though it has yet to find applications. This is likely due to the fact that $a$ cannot be chosen inside $\Fp$, so it is difficult to extract useful information from this single excising element. This is exactly the problem quantum P-points are made to solve, so let us define them now.

\begin{defn}\label{Pstuff}
    A \textit{P-filter} is a preordered set $(\mathscr{F},<)$ which is upwards closed and $\sigma$-downwards directed. A \textit{P-point} is an ultrafilter which is a P-filter under $\subseteq^*$, inclusion modulo finite sets. A \textit{quantum P-point} is a maximal quantum filter which is P-filter under $\ll$. That is, for any countable family $\mathscr{F}\subseteq\Fp$, there is $a\in\Fp$ such that $a\ll b$ for all $b\in\mathscr{F}$. We call $\varphi$ a \textit{P-state} if $\Fp$ is a quantum P-point. Likewise, we call a representation $\pi$ of $C$ a \textit{P-representation} if it is equivalent to the GNS-representation corresponding to a P-state.
\end{defn}

As mentioned above, the most important quality of quantum P-points is their strong excision property. Namely that for every separable $C^*$-subalgebra $B$ of $C$, $\varphi$ is excised on $B$ by a single $a\in\mathscr{F}_\varphi$. We call this property \textit{P-excision}. On an intuitive level, it is not hard to see why this would make a serious difference, as $a\in\Fp$ allows us to claim $|\varphi(b)|\leq\|aba\|$ for every $b\in C$, not only on the excised subalgebra. This will be crucial when realizing countable degree-1 types in theorem \ref{satequiv}.

\subsection{Equivalences}

Let us give some order-theoretic terms relevant to our main equivalence theorem.

\begin{defn}\label{order}
    A preordered set $(\mathscr{F},<)$ is \textit{countably closed} if every decreasing sequence in $\mathscr{F}$ has a lower bound. A subset $\mathscr{P}$ of $\mathscr{F}$ is \textit{cofinal} if for every $a\in\mathscr{F}$, there exists $b\in\mathscr{P}$ such that $b\leq a$. Lastly, a totally ordered set $(\mathscr{P},<)$ with no endpoints is an $\eta_1$-set if for all countable $X,Y\subseteq\mathscr{P}$ such that $a<b$ for all $a\in X$, $b\in Y$, there exists $c\in\mathscr{P}$ such that $a<c<b$ for all $a\in X$, $b\in Y$.
\end{defn}

In the following lemma to theorem \ref{satequiv}, we expose a new trick for realizing projections using degree-1 norm types in algebras of real rank zero. Before proceeding, let us define the \textit{quantum filter type}
\[\tF=\{\|x\|=1,x\geq0,\varphi(x)=1\},\]
a finite degree-1 type whose realizations are precisely $\Fp$.

\begin{lem}\label{cof}
    If $C$ has real rank zero, $\varphi$ is a state on $C$, and $(C,\varphi)$ is countably degree-$1$ saturated, then $\Pp$ is cofinal in $\Fp$.
\end{lem}

\begin{proof}
    Fix $a\in\Fp$ and the type
    \[t(x)=\tF\cup \{ax=x\}.\]
    By \cite[5.3.10]{FA19}, for each $\epsilon\in(0,1)$, there exists $q\in\Pp$ such that $q\leq a+\epsilon$, so $1-a\leq 1-q+\epsilon$ and
    \[q(1-a)q\leq q(1-q+\epsilon)q=\epsilon q\leq\epsilon.\]
    Hence
    \[\|(1-a)^\frac{1}{2}q\|^2=\|q(1-a)^\frac{1}{2}(1-a)^\frac{1}{2}q\|=\|q(1-a)q\|\leq\epsilon,\]
    and
    \[\|(1-a)q\|=\|(1-a)^\frac{1}{2}(1-a)^\frac{1}{2}q\|\leq\|(1-a)^\frac{1}{2}\|\epsilon^\frac{1}{2}\leq\epsilon^\frac{1}{2}.\]
    Therefore $t(x)$ is satisfiable, so realized by $b\in\Fp$ such that $b\ll a$. Then we repeat this with $t'(x)=\tF\cup\{bx=x\}$ to obtain $c\in\Fp$ such that $c\ll b\ll a$. Then by real rank zero and \cite[2.7.5]{FA19}, there exists a projection $p$ such that $c\ll p\ll a$, so $p\in\Pp$.
\end{proof}

\begin{lem}\label{psig}
    Let $A$ have real rank zero and $\varphi$ a state on $A$.
    If $\Pp$ is a P-filter, then $\Fp$ is a P-filter. In fact, lower bounds of countable families in $\Fp$ can be chosen to be projections.
\end{lem}

\begin{proof}
    Enumerate a sequence $(a_n)$ in $\Fp$. As noted in \cite[5.3.10]{FA19}, for all $n,m$, there exists $p_{nm}\in\Pp$ such that $p_{nm}\leq a_n+\frac{1}{m}$. Then by the P-filter property, fix $p\in\Pp$, a lower bound for every $p_{nm}$. Then for all $n,m$,
    \[p\ll p_{nm}\leq a_n+\frac{1}{m},\]
    so $p\ll a_n$. Since $(a_n)$ was arbitrary, we are done.
\end{proof}

We state these two lemmas separately from theorem \ref{satequiv}, as they apply to all states, not only pure ones. Also, lemma \ref{psig} does not require any saturation. Of particular interest to operator algebraists is the fact that these states have weak normality, a property first discovered for projections on the Calkin algebra by Bice in \cite[6.2]{BI13}. That is, decreasing sequences $(a_n)$ in $A_+$ have lower bounds $a\in A_+$ with $\varphi(a)=\lim_{n\rightarrow\infty}\varphi(a_n)$.

\begin{thm}\label{satequiv}
    Let $C$ be countably degree-$1$ saturated and $\varphi$ a pure state on $C$. The following are equivalent:
    \begin{enumerate}
        \item $(C,\varphi)$ is countably degree-1 saturated.
        \item $\Fp$ is a quantum P-point.
    \end{enumerate}
    If in addition $C$ has real rank zero, the following are equivalent to $(1),(2)$:
    \begin{enumerate}[resume]
        \item $\Pp$ is a P-filter.
        \item $\Pp$ is countably closed.
        \item $\Pp$ is cofinal in $\Fp$.
    \end{enumerate}
    If $C$ also has no minimal projections, the following is equivalent to $(1)$-$(5)$:
    \begin{enumerate}[resume]
        \item Maximal chains in $\Pp\setminus\{1\}$ are $\eta_1$-sets.
    \end{enumerate}
    If $C$ is also purely infinite and simple, the following is equivalent to $(1)$-$(6)$:
    \begin{enumerate}[resume]
        \item $\Pp$ is a filter.
    \end{enumerate}
\end{thm}

\begin{figure}[h]
        \centering
        \begin{tikzpicture}[>=stealth, node distance=1cm, scale=0.6]
    \foreach \i in {1,...,7} {
        \node (N\i) at ({90-360/7*(\i-1)}:1.8cm) {\i};
    }
    \draw[->] (N1) -- (N2);
    \draw[<->] (N4) -- (N6);
    \draw[<->] (N4) -- (N7);
    
    \draw[->] (N3) -- (N7);
    \draw[->] (N3) -- (N4);
    \draw[->] (N1) -- (N2);
    \draw[->] (N1) -- (N3);
    \draw[->] (N3) -- (N2);
    \draw[->] (N1) -- (N5);
    \draw[->] (N5) -- (N4);

    \draw[->, dashed, bend left=55]
        (N4) to (N3);
    \draw[->, dashed, bend right=20]
        (N7) to (N3);
    \draw[->, thick, dotted, bend right=20]
        (N2) to (N1);
\end{tikzpicture}
        \caption{Equivalences}
        \label{equiv}
    \end{figure}
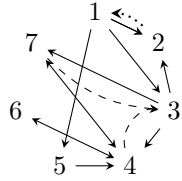

\begin{proof}
    Figure \ref{equiv} depicts the layout of equivalences. Firstly, note that $(3)\Leftrightarrow (4)\wedge(7)$, denoted by dashed arrows in figure \ref{equiv}, regardless of the assumptions on $C$, $(1)\Rightarrow (5)$ is the statement of \ref{cof}, and $(3)\Rightarrow (2)$ is the statement of \ref{psig}. We will prove the other directions in the following order: 
    \[(1) \Rightarrow (2),(3); (4)\Leftrightarrow (7);(5)\Rightarrow(4); (4)\Leftrightarrow(6)\]
    and we save $(2)\Rightarrow (1)$ for later, denoted by the dotted arrow in figure \ref{equiv}.
    
    For $(1)\Rightarrow (2)$, let $Y\subseteq\Fp$ be countable and define the type
    \[t(x)=\tF\cup\{ax=x:a\in Y\}.\]

    By \cite[5.3.4]{FA19}, the set $\mathscr{D}=\{a\in C_{+,1}:1-a\in \Lambda\}$ is dense in $\Fp$, where 
    \[\Lambda=\{b\in(L_\varphi\cap L_\varphi^*)_+:b<1\}.\]
    Now define an approximate unit $(e_\gamma)$ for $L_\varphi\cap L_\varphi^*$. Notice that for each $a\in \mathscr{D}$, $(1-e_\gamma)a-(1-e_\gamma) = (e_\gamma-1)(1-a)
        =e_\gamma(1-a)-(1-a).$ Now fix $F\Subset Y$, $\epsilon>0$, and for each $a\in F$, choose $a'\in\mathscr{D}$ such that $\|a'-a\|<\frac{\epsilon}{2}$. There exists $\gamma$ such that for all $a\in F$, $\|e_\gamma (1-a')-(1-a')\|<\frac{\epsilon}{2}$. Therefore
    \begin{align*}
        \|(1-e_\gamma)a-(1-e_\gamma)\| &\leq \|1-e_\gamma\|\|a-a'\|+\|e_\gamma(1-a')-(1-a')\| \\
        &<\frac{\epsilon}{2}+\frac{\epsilon}{2}=\epsilon
    \end{align*}
    so $(1-e_\gamma)$ satisfies $t(x)$, so it is realized by $b\in\Fp$ such that $b\ll a$ for all $a\in Y$. Since $Y$ was arbitrary, we are done.

    For $(1) \Rightarrow (3)$, fix countable $Y\subseteq\Pp\subseteq\Fp$. Then by (2), fix a lower bound $b\in\Fp$ for $Y$. Then by lemma \ref{cof} there exists $p\in\Pp$ such that $p\ll b$, so $p$ a lower bound for $Y$.

    For $(4) \Leftrightarrow (7)$, we apply both directions of \cite[4.1]{BI13}.

    For $(5)\Rightarrow (4)$, enumerate decreasing $(p_n)$ in $\Pp$. Then define $a\in C_{+,1}$ by
    \[a=\sum_{n=1}^\infty\frac{1}{n(n+1)}p_n\]
    and notice
    \[\varphi(a)=\sum_{n=1}^\infty\frac{1}{n(n+1)}=1,\]
    so $a\in\Fp$.
    By cofinality, fix some $p\in\Pp$ such that $p\ll a$. Then
    \begin{align*}
        p\leq a &= \sum_{n=1}^\infty\frac{1}{n(n+1)}p_n \\
        &\leq \sum_{n=1}^m\frac{1}{n(n+1)}+\sum_{n=m+1}^\infty\frac{1}{n(n+1)}p_n \\
        &\leq \bigg(1-\frac{1}{m}\bigg) +\frac{1}{m}p_m \\
        &\leq p_m+\bigg(1-\frac{1}{m}\bigg)(1-p_m)
    \end{align*}
    Since $1-\frac{1}{m}<1$, this is only possible if $p\leq p_m$, so $p\ll p_m$. Since $m$ and $(p_n)$ were arbitrary, $\Pp$ is countably closed.

    For (4)$\Leftrightarrow$(6), note that every case of the $\eta_1$ condition of $\mathscr{P}$ is realized by countable degree-1 types over $C$, except the case where $X$ is empty and $Y$ is countably infinite. This is to say, the condition that countably infinite $Y\subseteq\mathscr{P}$ have lower bounds in $\mathscr{P}$. Either $Y$ has a lower bound in $\mathscr{P}$ or it is cofinal in $\mathscr{P}$. Suppose $Y$ is cofinal and $\Pp$ is also countably closed, then $Y$ has a lower bound in $\Pp$, so by maximality, in $\mathscr{P}$. Thus $Y$ must contain the minimal element of $\mathscr{P}$, but $C$ has no minimal projections, so neither does $\Pp$, a contradiction. Thus, $(4)\Rightarrow(6)$. Conversely, every decreasing sequence in $\Pp$ is contained in some maximal chain, so $(6)\Rightarrow (4)$.
\end{proof}

Compare condition (6) with \cite[3]{HA98}, which states that maximal chains of projections in $\mathscr{Q}(H)$ are $\eta_1$-sets. By the above reasoning, this theorem also holds for any countably degree-1 saturated $C$ with with real rank zero and no minimal projections.

\begin{cor}\label{sqe}
Let $\varphi$ be a pure state such that $(C,\varphi)$ countably degree-$1$ saturated. Then $\varphi$ has P-excision. If in addition $C$ has real rank zero, then the excising element can always be chosen to be a projection.
\end{cor}

\begin{proof}
    Let $B$ be a separable $C^*$-subalgebra of $C$. As noted in \cite[15.3.7]{FA19}, Löwenheim-Skolem grants us a countable $Y\subseteq\Fp$ which excises $\varphi$ on $B$. By theorem \ref{satequiv} $(1)\Rightarrow (2)$, this subset has a lower bound $a\in\Fp$. Then as noted in \cite[15.3.8]{FA19}, for $b\in B$, $c\in Y$,
    \[\|aba-\varphi(b)a^2\|=\|a(b-\varphi(b))a\|=\|ac(b-\varphi(b))ca\|\leq \|a\|^2\|c(b-\varphi(b))c\|,\]
    so
    \[\|aba-\varphi(b)a^2\|\leq \|a\|^2\inf_{c\in Y}\|cbc-\varphi(b)c^2\|=0.\]
    Since $b$ was arbitrary, we are done.
\end{proof}

This P-excision proof actually gives us the crucial tool for proving $(2)\Rightarrow (1)$ of theorem \ref{satequiv}.

\begin{proof}[Proof of Theorem \ref{satequiv}]
    For $(2)\Rightarrow (1)$, fix a satisfiable degree-1 type
    \[t(\bar{x})=\{\rho_j(P_j(\bar{x}))=s_j:j\in\mathbb{N}\}.\]
    By the linearity of $\varphi$, for $\rho_j=\varphi$, we can replace $P_j(\bar{x})$ by $P_j(\bar{x})-d_j$, where $d_j\in A$ such that $\varphi(d_j)=s_j$. Thus, without loss of generality, when $\rho_j=\varphi$, assume $s_j=0$. Now fix, for each $m\in\mathbb{N}$, $\bar{b}_m$ such that $\max_{j\leq m}|\rho_j(P_j(\bar{b}_m))-s_j|<\frac{1}{m}.$
    Now let $X$ be the collection of entries of each $\bar{b}_m$ and coefficients of each $P_j$. Then $X$ is countable, so $B=C^*(X)$ is separable. Mimicking the proof of \ref{sqe}, there exists a single $a\in\Fp$ which excises $\varphi$ on $B$. Now define the modified *-polynomials 
    \[P'_j(\bar{x}) = \begin{cases} 
        aP_j(\bar{x})a, & \rho_j=\varphi; \\
        P_j(\bar{x}), & \rho_j=\|\cdot\|; 
        
   \end{cases}\]
   and fix the countable degree-$1$ norm type
   \[t'(\bar{x})=\{\|P'_j(\bar{x})\|=s_j:j\in\mathbb{N}\}\]
   and notice $(\bar{b}_m)$ satisfies $t'(\bar{x})$. 
   Since $C$ is countably degree-1 saturated, $t'(\bar{x})$ is realized. When $\rho_j=\varphi$, we have
   \[|\varphi(P_j(\bar{b}))|=|\varphi(aP_j(\bar{b})a)|\leq\|aP_j(\bar{b})a\|=0.\]
   Therefore, $\bar{b}$ realizes $t(\bar{x})$.
\end{proof}

Now that we have these equivalences, we will simply refer to pure $\varphi$ such that $(C,\varphi)$ is countably degree-1 saturated as a P-state.

\section{Infinite Transitivity}

The reader may have noticed a concerning gap between the countable degree-1 saturation of $(C,\varphi)$ and an infinite Kadison transitivity theorem: Vector equations in $H_\varphi$ are degree-2 conditions of the form 
\[0=\|xa-b\|^\varphi_2=\varphi((xa-b)^*(xa-b))^\frac{1}{2}.\] 
The following lemma shows the true potential of P-states and addresses these concerns.

\begin{lem}\label{fields}
    Let $C$ be countably degree-1 saturated. Then $C$ is saturated over countable types consisting of the following conditions
    \begin{enumerate}
        \item $\|P(\bar{x})\| \in K$
        \item $\|P(\bar{x})\|^\varphi_2 \leq r$
        \item $\varphi(P(\bar{x})) \in \Lambda$
    \end{enumerate}
    where $P(\bar{x})$ a degree-$1$ $*$-polynomial, $K\subseteq[0,\infty)$, $\Lambda\subseteq\mathbb{C}$ both compact, $r\geq 0$, and $\varphi$ a P-state.
\end{lem}

\begin{proof}
    Let $t(\bar{x})=\{\rho_j(P_j(\bar{x}))\in K_j:j\in\mathbb{N}\}$ be a satisfiable countable type consisting of conditions above. To prove that it is realized, we will perform two conversions of $t(\bar{x})$.
    
    Since $t(\bar{x})$ is satisfiable, fix $\bar{b}_m$ such that $\max_{j\leq m}d(\rho_j(P_j(\bar{b}_m)),K_j)\leq\frac{1}{m}$. Then fix an ultrafilter $\U$ on $\mathbb{N}$. Let $s_j=\lim_{m\rightarrow\U}\rho_j(P_j(\bar{b}_m))$. Then
    \[t'(\bar{x})=\{\rho_j(P_j(\bar{x}))=s_j:j\in\mathbb{N}\}\]
    is satisfiable type whose realizations realize $t(\bar{x})$. Now for each $\varphi$, choose $a_\varphi\in\Fp$ excising the $C^*$-subalgebra $B$ generated by coefficients of $P_j(\bar{x})$ and entries of $\bar{b}_m$. Then, for each $j$ such that $\rho_j=\varphi$, choose $d_j\in B$ such that $\varphi(d_j)=s_j$. Define
    \[P'_j(\bar{x}) = \begin{cases} 
        a_\varphi(P_j(\bar{x})-d_j)a_\varphi, & \rho_j=\varphi; \\
        P_j(\bar{x})a_\varphi, & \rho_j=\|\cdot\|_2^\varphi; \\
        P_j(\bar{x}), & \rho_j=\|\cdot\|;
   \end{cases}\]

   \[r_j = \begin{cases} 
        0, & \rho_j=\varphi; \\
        s_j, & \text{otherwise}.
   \end{cases}\]
   
   The type
   \[t''(\bar{x})=\{\|P'_j(\bar{x})\|=r_j:j\in\mathbb{N}\}\]
   is a satisfiable countable degree-1 norm type whose realizations realize $t(\bar{x})$. To see this, we only need to check the case where $\rho_j=\|\cdot\|_2^\varphi$,
   \[\|P_j(\bar{b})\|_2^\varphi=\varphi(P_j(\bar{b})^*P_j(\bar{b}))^\frac{1}{2}\leq\|a_\varphi^*P_j(\bar{b})^*P_j(\bar{b})a_\varphi\|^\frac{1}{2}=\|P_j(\bar{b})a_\varphi\|=r_j\in K_j.\]
   In this case, $K_j$ is of the form $[0,r]$, so $\|P_j(\bar{b})\|^\varphi_2\in K_j$.
   Therefore, since $t''(\bar{x})$ is realized, so is $t(\bar{x})$.
\end{proof}

We now prove the main results of the paper, infinite Glimm-Kadison for direct sums of P-representations of countably degree-1 saturated $C^*$-algebras, from which infinite transitivity will follow.

\begin{proof}[Proof of Theorem \ref{igk}]
    Let $K_j$ be a separable subspace of $H_j$ and enumerate an orthonormal basis $(\xi_i)$ of $K=\bigoplus_jK_j$ compatible with each $K_j$. Every $T\in\prod_j\mathscr{B}(K_j)$ is determined by $(T\xi_i)$. For each $i,j\in\mathbb{N}$, lift $a_{ij}+L_{\varphi_j}=(\xi_i)_j$ and $b_{ij}+L_{\varphi_j}=(T\xi_i)_j$.
    Define the type
    \[t(x)=\{\|x\|\leq\|T\|,\|xa_{ij}-b_{ij}\|^{\varphi_j}_2=0:i,j\in\mathbb{N}\}.\]
    Since $t(x)$ consists of conditions from \ref{fields}, it suffices to show that it is satisfiable. Fix $m\in\mathbb{N}$. Because $\{\xi_j:j\leq m\}$ is orthogonal, by the Glimm-Kadison theorem, for each $\epsilon>0$ there exists $d\in C$ such that $\pi_j(d)\xi_{i}=T\xi_i$ for all $i,j\leq m$ and $\|d\|<\|T\|+\epsilon$. That is,  $\|da_{ij}-b_{ij}\|^{\varphi_j}_2=0,$ for all $j\leq m$ and $\|d\|<\|T\|+\epsilon$. Therefore $t(x)$ is satisfiable and is realized by $c\in C$. Finally, $\|c\|\geq\|\pi(c)\|\geq\|T\|$, so $\|c\|=\|T\|$ and $P_K\pi(c)P_k=T$. 
    
    For self-adjoint or positive, add the necessary generalized degree-$1$ norm conditions to $t(x)$ \cite[15.2.9]{FA19}, yielding a satisfiable type which is realized by countably degree-1 saturated $C$, thanks to \ref{fields}.

    For unitary, notice that by the Borel functional calculus, $T$ has self-adjoint complex logarithm $S\in\prod_j\mathscr{B}(K_j)$. Then, by the self-adjoint case, there exists self-adjoint $s\in C$ so that $P_K\pi(s)P_K=S$, so $P_K\pi(e^{is})P_K=T$, so $c=e^{is}$ is our desired unitary.
\end{proof}

We note that we can readily produce infinite versions of \cite[3.5.5, 3.8.3, 3.9.6, 5.1.3, 5.6.6]{FA19} in the case of P-states/representations, but from here these proofs are so elementary that they would only serve to dilute this text.

\section{Existence of P-states}

The conclusion of theorem \ref{satequiv} can be leveraged to show it is relatively consistent with ZFC that there are no P-states on many corona algebras. However, in any model of ZFC which admits a P-point (an extremely weak assumption), we can produce a P-state on many coronas. To accomplish this, we want to mimic the construction of diagonal pure states on $\mathscr{Q}(H)$. As such, we need a generalization of an orthonormal sequence of vectors.

\begin{defn}\label{Pbuilder}
    Let $A$ be nonunital, $\sigma$-unital. We call a sequence $(q_j)$ of projections a \textit{sub-$\sigma$ sequence of $A$} if there is a $\sigma$-unit $(e_j)$ of $A$ such that $q_j\leq e_j-e_{j-1}$, $e_{-1}=0$. For a sub-$\sigma$ sequence, define the map $p:\mathscr{P}(\mathbb{N})\rightarrow\Proj(\mathscr{M}(A))$ by $p_X$ equal to the strict limit of the sequence
    \[p_{X,n}=\sum_{\substack{j\in X \\ j\leq n}}q_j\]
    and we abuse the notation $p_X=\sum_{j\in X}q_j$.
    For each $\varphi\in S(\mathscr{Q}(A))$, define
    \[\Xp=\{X\subseteq\mathbb{N}:\varphi(\pi(p_X))=1\}.\]

    We call a sub-$\sigma$ sequence \textit{scalar} if $q_j$ is a scalar projection for all $j$ and there is a faithful, nondegenerate representation $\rho:A\rightarrow \mathscr{B}(H)$ such that $\rho(q_j)$ has finite rank for all $j$.
\end{defn}

To justify the definition of a scalar sub-$\sigma$ sequence, let $A=\bigoplus_{i\in I}\mathscr{K}(H_i)$ such that $\bigoplus_{i\in I}H_i$ is infinite dimensional, separable. Then $A$ has a scalar sub-$\sigma$ sequence. Namely, we identify $A\subseteq\mathscr{B}(\bigoplus_i H_i)$, fix an orthonormal basis $(\xi_j)$ of $\bigoplus_iH_i$ compatible with the direct sum, making $q_j=\text{proj}_{\mathbb{C}\xi_j}$ scalar sub-$\sigma$. Any subsequence of a scalar sub-$\sigma$ sequence is still scalar sub-$\sigma$.

The following lemma was originally proven for the Calkin algebra by Zamora-Avilés \cite[2.5.10]{ZA10}. Here we provide a version generalized to coronas of algebras with a sub-$\sigma$ sequence.

\begin{lem}\label{beatr}
    Let $A$ be a nonunital, $\sigma$-unital $C^*$-algebra. For every sub-$\sigma$ sequence in $A$, decreasing sequence $X_n\subseteq\mathbb{N}$, and $a\in \mathscr{Q}(A)_{+,1}$, we have $a\ll \pi(p_{X_n})$ for all $n$ if and only if there exists a pseudointersection $X\subseteq\mathbb{N}$ of $(X_n)$ such that $a\ll \pi(p_X)$.
\end{lem}

\begin{proof}
    Without loss of generality, suppose $X_1=\mathbb{N}$ and $X_n\setminus X_{n+1}$ infinite. Enumerate $X_n\setminus X_{n+1}=\{j^n_i:i\in\mathbb{N}\}$. Now define $p^n:\mathscr{P}(\mathbb{N})\rightarrow \Proj(\mathscr{M}(A))$ by $p^n_Y=p_{\{j^n_i:i\in Y\}}$. Fix a lift $\tilde{a}\in\mathscr{M}(A)$, so $a=\pi(\tilde{a})\ll\pi(p_{X_n})$, and we have 
    \[\pi(p_{X_n\setminus X_{n+1}}\tilde{a})=(\pi(p_{X_n})-\pi(p_{X_{n+1}}))a=a-a=0.\]
    By the sub-$\sigma$ property, for each $m$, there is $k$ such that $(1-e_k)p^n_{\mathbb{N}}=p^n_{\mathbb{N}\setminus m}$. This yields
    \[\lim_{k\rightarrow\infty}\|p^n_{\mathbb{N}\setminus k}\tilde{a}\|=\lim_{k\rightarrow\infty}\|(1-e_k)p_{X_n\setminus X_{n+1}}\tilde{a}\|=0,\]
    so the function $f_\epsilon\in\mathbb{N}^\mathbb{N}$ by
    \[f_\epsilon(n)=\min\Big\{k:\|p^n_{\mathbb{N}\setminus k}\tilde{a}\|<\frac{\epsilon}{2^{n+1}}\Big\}\]
    is well-defined. Fix $f\in\mathbb{N}^\mathbb{N}$ such that $f_\frac{1}{m}\leq^*f$ for all $m$. That is, $f_\frac{1}{m}(n)\leq f(n)$ for eventually all $n$. Now we claim that $X=\bigcup_{n\in\mathbb{N}}\{j^n_i:i\leq f(n)\}$
    is our desired pseudointersection. Since $X_n$ decreases from $\mathbb{N}$, $X_n\setminus X_{n+1}$ partitions $\mathbb{N}$, so
    \[X\setminus X_n=\bigcup_{j=1}^{n-1}X\cap(X_j\setminus X_{j+1})=\bigcup_{j=1}^{n-1}\{j^n_i:i\leq f(n)\},\]
    so
    \[|X\setminus X_n|=\sum_{j=1}^{n-1}f(j)<\infty,\]
    thus $X$ is a pseudointersection. Now we claim $a\ll\pi(p_X)$, and it suffices to show that $\|(1-e_k)p_{\mathbb{N}\setminus X}\tilde{a}\|\rightarrow 0$. By $f_\frac{1}{m}\leq^*f$, we have
    \begin{align*}
        \|(1-e_k)p_{\mathbb{N}\setminus X}\tilde{a}\| &\leq  \sum_{n=1}^\infty\|(1-e_k)p^n_{\mathbb{N}\setminus f(n)}\tilde{a}\| \\
        &\leq \sum_{n=1}^{N-1}\|(1-e_k)p^n_{\mathbb{N}\setminus f(n)}\tilde{a}\|+\sum_{n=N}^\infty\frac{1}{m2^{n+1}} \\
        &\leq \sum_{n=1}^{N-1}\|(1-e_k)p^n_{\mathbb{N}\setminus f(n)}\tilde{a}\|+\frac{1}{m2^N} \\
        &\leq \sum_{n=1}^{N-1}\|(1-e_k)p_{X_n\setminus X_{n+1}}\tilde{a}\|+\frac{1}{m} \\
        &\leq \frac{1}{m}+\frac{1}{m}=\frac{2}{m}
    \end{align*}
    by taking $k$ so that $\|(1-e_k)p_{X_n\setminus X_{n+1}}\tilde{a}\|\leq\frac{1}{(N-1)m}$ for all $n< N$.
\end{proof}

This lemma shows us that P-states cannot naturally arise from non-P-points.

\begin{cor}\label{Pup}
    Let $A$ be nonunital, $\sigma$-unital. For every sub-$\sigma$ sequence, there exists a natural $\varphi:\beta\mathbb{N}\setminus\mathbb{N}\hookrightarrow P(\mathscr{Q}(A))$ such that $\Xp=\mathscr{U}$ and $\varphi_\mathscr{U}$ is a P-state only if $\U$ is a P-point.
\end{cor}

\begin{proof}
    Let $(q_j)$ be a sub-$\sigma$ sequence. Fix $\varphi_j\in P(A)$ such that $\varphi_j(q_j)=\|q_j\|=1$. By \cite[3.3.9]{MU90}, these each have unique extensions $\tilde{\varphi_j}\in S(\mathscr{M}(A))$ and $\tilde{\varphi_j}$ is pure because $q_j$ is still scalar in $\mathscr{M}(A)$. Now define $\varphi:\beta\mathbb{N}\setminus\mathbb{N}\hookrightarrow S(\mathscr{Q}(A))$ by $\varphi_\U(a)=\lim_{j\rightarrow\U}\tilde{\varphi_j}(\tilde{a})$, for any lift $\tilde{a}\in\mathscr{M}(A)$ of $a$. We know these are states because they lift to weak$^*$-limits of pure states on $\mathscr{M}(A)$ and they are pure by \cite[1]{AN79}.

    Suppose $\U$ is not a P-point, and fix decreasing $X_n\in\U$ with no pseudointersection in $\U$. Our claim is that $\pi(p_{X_n})$ has no lower bound in $\mathscr{F}_{\varphi_\U}$. By \ref{beatr}, any lower bound $a\in\Fp$ of $(\pi(p_{X_n}))$ must in fact have $a\ll \pi(p_X)$, where $X$ is a pseudointersection of $(X_n)$, therefore no such $a$ exists.
\end{proof}

The following theorem gives us a method for finding P-states.

\begin{thm}\label{stoneinj}
    Let $A$ be nonunital, $\sigma$-unital. For every scalar sub-$\sigma$ sequence in $A$, there exists a canonical $\varphi:\beta\mathbb{N}\setminus\mathbb{N}\hookrightarrow P(\mathscr{Q}(A))$ such that $\Xp=\mathscr{U}$ and $\varphi_\mathscr{U}$ is a P-state if and only if $\U$ is a P-point. In particular, if $\U$ is a P-point, lower bounds of countable families in $\mathscr{F}_{\varphi_\U}$ can be chosen of the form $\pi(p_X)$ for $X\in\U$.
\end{thm}

\begin{proof}
    Let $(q_j)$ be a scalar sub-$\sigma$ sequence. As $q_j$ is a scalar projection, there is a unique $\varphi_j\in P(A)$ such that 
    \[\mathscr{F}_{\varphi_j}=\{a\in A_{+,1}:q_j\leq a\}.\]
    We define $\varphi:\beta\mathbb{N}\setminus\mathbb{N}\hookrightarrow P(\mathscr{Q}(A))$ as in \ref{Pup}. By \ref{Pup}, if $\U$ is not a P-point, $\varphi_\U$ is not a P-state

    Conversely, suppose $\mathscr{U}$ is a P-point. By the scalar sub-$\sigma$ property, fix a faithful, nondegenerate representation $\rho:A\rightarrow\mathscr{B}(H)$ such that $\rho(q_j)$ is finite dimensional. There is a unique extension $\tilde{\rho}:\mathscr{M}(A)\rightarrow\mathscr{B}(H)$, per \cite[13.2.1]{FA19}. Fix an orthonormal basis $(\xi_\gamma:\gamma<\kappa)$ of $H$ which simultaneously diagonalizes $\tilde{\rho}(q_j)$ for every $j$ and express
    $\tilde{\rho}(q_j)=\text{proj}_{\overline{\text{span}}\{\xi_\gamma:\gamma\in\Lambda_j\}}$ for some $\Lambda_j\Subset\kappa$.

    Fix $a\in\mathscr{F}_{\varphi_\U}$ and a lift $\tilde{a}\in\mathscr{M}(A)$ such that $\|\tilde{a}\|=1$. Now let 
    \[X_n=\big\{k\in\mathbb{N}:\tilde{\varphi_k}(\tilde{a})>1-\frac{1}{n}\big\}.\]
    Since $\mathscr{U}$ is a P-point, there is $X\in\mathscr{U}$ such that $X\subseteq^* X_n$. We first claim that
    \[\|\pi(\rho(p_{X}(1-\tilde{a})p_{X}-\sum_{j\in {X}}q_j(1-\tilde{a})q_j))\|=0,\]
    where $\pi:\mathscr{B}(H)\rightarrow\mathscr{Q}(H)$ is the quotient map. Suppose this is not the case. Then there exists $\epsilon>0$ and infinitely many $\gamma<\kappa$ such that
    \[\|\tilde{\rho}\big(\sum_{\substack{i,j\in X \\ i\neq j}}q_i\tilde{a}q_j\big)\xi_\gamma\|\geq\epsilon.\]
    Some linear algebra yields
    \[\|\tilde{\rho}\big(\sum_{\substack{i,j\in X \\ i\neq j}}q_i\tilde{a}q_j\big)\xi_\gamma\|^2=\sum_{\substack{i\in X \\ i\neq j}}\|\tilde{\rho}(q_i\tilde{a}q_j)\xi_\gamma\|^2,\]
    when $\gamma\in \Lambda_j$ and $0$ otherwise. Now,
    \begin{align*}
    \|\tilde{\rho}(p_{X}\tilde{a}p_{X})\xi_\gamma\|^2 &=\sum_{i\in X}\|\tilde{\rho}(q_i\tilde{a}q_j)\xi_\gamma\|^2 \\
    &\geq \|\tilde{\rho}(q_j\tilde{a}q_j)\xi_\gamma\|^2+\epsilon^2 \\
    &\geq \tilde{\varphi_j}(\tilde{a})^2\|\tilde{\rho}(q_j)\xi_\gamma\|^2+\epsilon^2 \\
    &\geq \tilde{\varphi_j}(\tilde{a})^2+\epsilon^2
    \end{align*}
    when $\gamma\in\Lambda_j$ as in the assumption. Since there are infinitely many such $\gamma$, we can take $\gamma\in \Lambda_j$ with $j\in X\setminus k$ for arbitrarily large $k$, \[\|\tilde{a}\|^2\geq\|\tilde{\rho}(p_{X}\tilde{a}p_{X})\|^2\geq\|\tilde{\rho}(p_{X}\tilde{a}p_{X})\xi_\gamma\|^2\geq \big(1-\frac{1}{n}\big)^2+\epsilon^2>1\]
    for sufficiently large $n$, a contradiction. Thus, $\tilde{\rho}(p_X(1-\tilde{a})p_X-\sum_{j\in X}q_j(1-\tilde{a})q_j)\in \mathscr{K}(H)$. Now, because $e_k\rightarrow 1_{\mathscr{B}(H)}$ strongly, $p_X(1-\tilde{a})p_X-\sum_{j\in X}q_j(1-\tilde{a})q_j\in A$, so back in $\mathscr{M}(A)$ with $\pi:\mathscr{M}(A)\rightarrow\mathscr{Q}(A)$ the quotient map
    \begin{align*}
        \|\pi(p_X(1-\tilde{a})p_X)\| &= \|\pi(\sum_{j\in X}q_j(1-\tilde{a})q_j)\| \\
        &= \lim_{k\rightarrow\infty}\|(1-e_k)\sum_{j\in X}q_j(1-\tilde{a})q_j\| \\
        &= \lim_{k\rightarrow\infty}\|\sum_{j\in X\setminus k}q_j(1-\tilde{a})q_j\|\\
        &= \lim_{k\rightarrow\infty}\sup_{j\in X\setminus k}\|q_j(1-\tilde{a})q_j\| \\
        &= \lim_{k\rightarrow\infty}\sup_{j\in X\setminus k}\tilde{\varphi_j}(1-\tilde{a}) =0.
    \end{align*}
    Therefore,
    \begin{align*}
        \|(1-a)\pi(p_X)\|^2 &\leq \|(1-a)^\frac{1}{2}\pi(p_x)\|^2 \\
        &\leq \|\pi(p_X(1-\tilde{a})p_X)\|=0,
    \end{align*}
    so $\pi(p_X)\ll a$.

    Finally, fix $(a_n)$ in $\mathscr{F}_{\varphi_\U}$ and $X_n\in\mathscr{U}$ so that $\pi(p_{X_n})\ll a_n$. Since $\U$ is a P-point, there is $X\in\U$ with $X\subseteq^* X_n$. Therefore $\pi(p_X)\ll \pi(p_{X_n})\ll a_n$.
\end{proof}

The existence of a P-point is extremely common in models of ZFC. In fact, until recently \cite{CH19}, there were very few known models of ZFC which have no P-points. Their existence is implied by the continuum hypothesis, is consistent with various standard forcing axioms, and even $\text{OCA}_\text{T}$. For the operator algebraist reader who does not fancy the continuum hypothesis, rest assured that P-states are likely relatively consistent with whichever axioms one fancies most.

\begin{lem}\label{extremepain}
    Let $A$ be nonunital, $\sigma$-unital with a sub-$\sigma$ sequence, and $\varphi\in P(\mathscr{Q}(A))$. If $\varphi$ is a P-state, then $\Xp$ is a non-meagre P-filter. Also, for any decreasing $X_n\subseteq\mathbb{N}$, there exists a pseudointersection $X$ of $(X_n)$ such that $\varphi(\pi(p_X))=\inf_n\varphi(\pi(p_{X_n}))$. Also, if $\varphi(\pi(p_{X_n}))\not\rightarrow0$, there exists $(k_n)$ in $\mathbb{N}$ increasing and $X\subseteq\mathbb{N}$ such that $X\subseteq^*\bigcup_{j\in [k_n,k_{n+1})}X_j$ and $\varphi(X)>0$.
\end{lem}

\begin{proof}
    We mimic the proof of \cite[6.2]{BI13}. To show that $\Xp$ is non-meagre, take an arbitrary interval partition $(I_n)$ of $\mathbb{N}$ and let $B$ be the $A$-strict completion of the $C^*$-subalgebra generated by $\{p_{I_n}:n\in\mathbb{N}\}$. This is isomorphic to $\ell^\infty$ and its quotient $B$ is isomorphic to $\ell^\infty/c_0$. Then $\Proj(B)$ is countably closed and atomless, so by \cite[6.1]{BI13}, $\Pp^B\neq\{1\}$. Therefore, $\Xp$ is non-meagre.

    Using the obvious type, there exists $a\in\Pp$ such that $a\ll \pi(p_{X_n})$ and $\varphi(a)=\inf_n\varphi(\pi(p_{X_n}))$. Then by \ref{beatr}, there is a pseudointersection $X$ of $(X_n)$ such that $a\ll\pi(p_X)$, so $1=\varphi(a)\leq\varphi(\pi(p_X))$. It then follows that $\Xp$ is a P-filter.

    For the last part of the claim, let $\epsilon=\limsup_{n\rightarrow\infty}\frac{\varphi(\pi(p_{X_n}))}{2}>0$. If \[\varphi(\pi(p_{X_{m_{i+1}}\setminus\bigcup_{j\in[k,m_i)}X_j}))\geq\frac{\epsilon}{3^{n+1}}\] for $i<l$, where $k<m_0<\cdots<m_l$, then $l\epsilon<3^{n+1}$. We may then choose $k_n$ such that $\varphi(\pi(p_{X_j}))>\epsilon$ for some $k_n\leq j< k_{n+1}$ and \[\varphi(\pi(p_{X_m\setminus\bigcup_{j\in[k_n,k_{n+1})}X_j}))<\frac{\epsilon}{3^{n+1}}\] for all $m\geq k_{n+1}$. Then
    \[\varphi(\pi(p_{\bigcap_{n\leq m}\bigcup_{j\in [k_n,k_{n+1})}X_j}))>\epsilon-\sum_{n<m}\frac{\epsilon}{3^{n+1}}\geq \epsilon -\frac{\epsilon}{2}=\frac{\epsilon}{2}\]
    By the normality ensured in the second paragraph, we have $X\subseteq^*\bigcup_{j\in[k_n,k_{n+1})}X_j$ so that $\varphi(\pi(p_X))\geq\frac{\epsilon}{2}>0$.
\end{proof}

\begin{thm}\label{corsat}
    Both of the following statements are relatively consistent with ZFC.
    \begin{enumerate}
        \item For every nonunital, $\sigma$-unital $A$ admitting a scalar sub-$\sigma$ sequence, there exist $2^\mathfrak{c}$-many inequivalent P-states $\varphi$ on $\mathscr{Q}(A)$.
        \item For every nonunital, $\sigma$-unital $A$ admitting a sub-$\sigma$ sequence, there exist no P-states on $\mathscr{Q}(A)$.
    \end{enumerate}
\end{thm}

\begin{proof}
    For (1), under the continuum hypothesis, there are $2^\mathfrak{c}$ RK-inequivalent selective ultrafilters (which are P-points) on $\mathbb{N}$, so the map $\varphi:\beta\mathbb{N}\setminus\mathbb{N}\hookrightarrow P(\mathscr{Q}(A))$ in \ref{stoneinj} produces $2^\mathfrak{c}$ inequivalent P-states.

    For (2), we use the forcing argument outlined in lemma \cite[6.5]{BI13}, where we use \ref{extremepain} in place of \cite[6.2]{BI13}. In this model, there are no $\varphi$ on $\mathscr{Q}(A)$ having the conclusion of \ref{extremepain}, so there are no P-states.
\end{proof}

\subsection{Counterexample to Infinite Transitivity}

The question of whether infinite transitivity in $(\mathscr{Q}(H),\pi_\varphi)$ implies degree-1 saturation of $(\mathscr{Q}(H),\varphi)$ is fairly opaque. In this section, we at least prove that in ZFC, the Calkin algebra has an irreducible representation which does not admit infinite transitivity.

\begin{defn}\label{idemp}
    For $X\subseteq\mathbb{N}$, $j\in\mathbb{N}$, we define $X+j=\{m+j:m\in X\}$ and for an ultrafilter $\U$ on $\mathbb{N}$, $\U+j=\{X+j:X\in \U\}$. Since $\beta\mathbb{N}$ is compact, we can define $\U+\mathscr{V}=\lim_{j\rightarrow\mathscr{V}}(\U+j)$. We call a nonprincipal ultrafilter $\U$ \textit{idempotent} if $\U+\U=\U$.
\end{defn}

    By \cite[1]{EL58}, every compact topological space with a semigroup structure whose left operation is continuous admits an idempotent element. The addition defined in \ref{idemp} plainly grants this semigroup structure on $\beta\mathbb{N}$, so there is $\U\in\beta\mathbb{N}$ idempotent, and this cannot be principal.

\begin{proof}[Proof of Theorem \ref{counter}]
    Let $\U$ be an idempotent ultrafilter on $\mathbb{N}$. We claim that $\pi=\pi_{\varphi_\U}$ fails infinite transitivity. Let $s\in\mathscr{B}(H)$ be the left shift along an orthonormal basis $(\eta_n)$ of $H$ and $\varphi_\U$ as in \ref{stoneinj} along $(\eta_n)$. Let $\xi_j=\pi(s^j)+L_{\varphi_\U}$. Then the equations
    \[\pi(a)\xi_j=\delta_{j0}\xi_j\]
    imply
    \[\varphi(\pi(s^j)^*a\pi(s^j))=\lim_{n\rightarrow\U}a_{n-j}=\delta_{j0}\]
    where $a_n=\langle a\eta_n,\eta_n\rangle$. Thus $X_1=\{n\in\mathbb{N}:|a_n-1|<\epsilon\}\in\U$ and for all $j>1$, $X_j=\{n\geq j:|a_{n-j}|<\epsilon\}\in\U$. However, $X_j=(X_1+j)\setminus j$, so by idempotency, $\{j\in\mathbb{N}:X_1+j\in\U\}\in \U$, so $\varphi(\pi(s^j)^*a\pi(s^j))=1$ for some $j\in\mathbb{N}$. Therefore, there is no $a\in\mathscr{Q}(H)$ such that $\pi(a)\xi_j=\delta_{j0}\xi_j$, so infinite transitivity fails.
\end{proof}

This tells us that infinite transitivity is not a property intrinsic to all irreducible representations of countably degree-1 saturated $C^*$-algebras, so the infinite transitivity of P-representations comes from the nontrivial interplay between the state and the algebra.

\end{document}